\documentclass[a4paper,12pt]{article}
\usepackage{amsmath}
\usepackage{amsthm}
\usepackage{amssymb}
\usepackage{enumerate}
\usepackage{url}
\usepackage[utf8]{inputenc}
\usepackage{comment}

\title{Invariance of Brownian motion associated with exponential functionals}

\author{Yuu Hariya\thanks{Supported in part by JSPS KAKENHI Grant Numbers 17K05288, 22K03330}}

\date{\empty}
\numberwithin{equation}{section}

\theoremstyle{plain}

\newtheorem{thm}{Theorem}[section]

\newtheorem{prop}{Proposition}[section]

\newtheorem{cor}{Corollary}[section]

\newtheorem{lem}{Lemma}[section]

\theoremstyle{definition}

\theoremstyle{remark}
\newtheorem{rem}{Remark}[section]

\begin{document}

\newcommand\ND{\newcommand}
\newcommand\RD{\renewcommand}

\ND\N{\mathbb{N}}
\ND\R{\mathbb{R}}
\ND\Q{\mathbb{Q}}
\ND\C{\mathbb{C}}

\ND\F{\mathcal{F}}

\ND\kp{\kappa}

\ND\ind{\boldsymbol{1}}

\ND\al{\alpha }
\ND\la{\lambda }
\ND\La{\Lambda }
\ND\ve{\varepsilon}
\ND\Om{\Omega}

\ND\ga{\gamma}

\ND\lref[1]{Lemma~\ref{#1}}
\ND\tref[1]{Theorem~\ref{#1}}
\ND\pref[1]{Proposition~\ref{#1}}
\ND\sref[1]{Section~\ref{#1}}
\ND\ssref[1]{Subsection~\ref{#1}}
\ND\aref[1]{Appendix~\ref{#1}}
\ND\rref[1]{Remark~\ref{#1}} 
\ND\cref[1]{Corollary~\ref{#1}}
\ND\eref[1]{Example~\ref{#1}}
\ND\fref[1]{Fig.\ {#1} }
\ND\lsref[1]{Lemmas~\ref{#1}}
\ND\tsref[1]{Theorems~\ref{#1}}
\ND\dref[1]{Definition~\ref{#1}}
\ND\psref[1]{Propositions~\ref{#1}}
\ND\rsref[1]{Remarks~\ref{#1}}
\ND\sssref[1]{Subsections~\ref{#1}}

\ND\pr{\mathbb{P}}
\ND\ex{\mathbb{E}}
\ND\br{W}

\ND\prb[2]{P^{(#1)}_{#2}}
\ND\exb[2]{E^{(#1)}_{#2}}

\ND\eb[1]{e^{B_{#1}}}
\ND\ebm[1]{e^{-B_{#1}}}
\ND\hbe{\Hat{\beta}}
\ND\hB{\Hat{B}}
\ND\cB{\Tilde{B}}
\ND\cA{\Tilde{A}}
\ND\argsh{\mathrm{Argsh}\,}
\ND\zmu{z_{\mu}}
\ND\GIG[3]{\mathrm{GIG}(#1;#2,#3)}
\ND\gig[3]{I^{(#1)}_{#2,#3}}
\ND\argch{\mathrm{Argch}\,}

\ND\vp{\varphi}
\ND\eqd{\stackrel{(d)}{=}}
\ND\db[1]{B^{(#1)}}
\ND\dcb[1]{\cB ^{(#1)}}
\ND\da[1]{A^{(#1)}}
\ND\dca[1]{\cA ^{(#1)}}
\ND\dz[1]{Z^{(#1)}}
\ND\Z{\mathcal{Z}}

\ND\anu{\alpha ^{(\nu )}}

\ND\Ga{\Gamma}
\ND\calE{\mathcal{E}}
\ND\calD{\mathcal{D}}

\ND\f{F}
\ND\g{G}

\ND\tr{\mathbb{T}}

\ND\ct{\mathcal{T}}

\ND\T{T}

\ND\id{\mathrm{Id}}

\ND{\rmid}[1]{\mathrel{}\middle#1\mathrel{}}

\def\thefootnote{{}}

\maketitle 
\begin{abstract}
It is well known that Brownian motion enjoys several 
distributional invariances such as the scaling property and 
the time reversal. In this paper, we prove another invariance 
of Brownian motion that is compatible with the time reversal. 
The invariance, which seems to be new to our best knowledge, 
is described in terms of an anticipative path transformation 
involving exponential functionals as anticipating factors. 
Some related results are also provided.
\footnote{Mathematical Institute, Tohoku University, Aoba-ku, Sendai 980-8578, Japan}
\footnote{E-mail: hariya@tohoku.ac.jp}
\end{abstract}

%%%%%% New section %%%%%%
\section{Introduction}\label{;intro}

Let $B=\{ B_{t}\} _{t\ge 0}$ be a one-dimensional standard 
Brownian motion and, for every $\mu \in \R $, denote by 
$\db{\mu }=\bigl\{ \db{\mu }_{t}:=B_{t}+\mu t\bigr\} _{t\ge 0}$ the 
Brownian motion with drift $\mu $. We define 
$\bigl\{ \da{\mu }_{t}\bigr\} _{t\ge 0}$ to be the quadratic variation of 
the geometric Brownian motion $\bigl\{ e^{\db{\mu }_{t}}\bigr\} _{t\ge 0}$, 
namely, 
\begin{align*}
 \da{\mu }_{t}:=\int _{0}^{t}e^{2\db{\mu }_{s}}ds.
\end{align*}
These exponential functionals of Brownian motion have 
importance in a variety of fields in probability theory such as 
mathematical finance, diffusion processes in random media, 
and probabilistic studies of Laplacians on hyperbolic spaces; 
see the detailed surveys \cite{mySI, mySII} by Matsumoto and Yor.

Let $C([0,\infty );\R )$ be the space of 
continuous functions $\phi :[0,\infty )\to \R $, on which we 
define 
\begin{align*}
 A_{t}(\phi ):=\int _{0}^{t}e^{2\phi _{s}}\,ds,\quad t\ge 0, 
\end{align*}
so that $\da{\mu }_{t}=A_{t}(\db{\mu }),\,t\ge 0$. When 
$\mu =0$, with slight abuse of notation, we simply write 
$A_{t}$ for $\da{0}_{t}$. For every fixed $t>0$, we 
have introduced in \cite{har22} the family 
$\{ \tr ^{t}_{z}\} _{z\in \R }$ of anticipative 
path transformations defined by 
\begin{align}\label{;ttrans}
 \tr ^{t}_{z}(\phi )(s):=\phi _{s}-\log \left\{ 
 1+\frac{A_{s}(\phi )}{A_{t}(\phi )}\left( e^{z}-1\right) 
 \right\} ,\quad 0\le s\le t, 
\end{align}
mapping the space $C([0,t];\R )$ of real-valued 
continuous functions $\phi $ over $[0,t]$ into itself. 
If there is no risk of ambiguity, we suppress the superscript $t$ 
from the notation and suppose that each $\tr _{z}$ acts on 
$C([0,t];\R )$. Let $\beta =\{ \beta (s)\} _{s\ge 0}$ be another 
one-dimensional standard Brownian motion that is independent 
of $B$. In \cite{har22}, we have shown the following identity in law 
which exhibits an invariance of the law of Brownian motion 
in the presence of the independent element $\beta $:  
for every $x\in \R $, the process 
\begin{align}\label{;ttrans1}
 \tr _{x+B_{t}-\argsh (e^{B_{t}}\sinh x +\beta (A_{t}))}
 (B)(s),\quad 0\le s\le t, 
\end{align}
is identical in law with $\{ B_{s}\} _{0\le s\le t}$. Here 
\begin{align*}
 \argsh x\equiv \log \Bigl( x+\sqrt{1+x^{2}}\Bigr) ,\quad x\in \R , 
\end{align*}
is the inverse function of the hyperbolic sine function. 
The above invariance extends Bougerol's celebrated identity 
in law (\cite{bou}): 
\begin{align}
 \beta (A_{t})\eqd \sinh B_{t}; \label{;boug}
\end{align}
indeed, evaluating \eqref{;ttrans1} at $s=t$ and taking $x=0$ 
leads to $\argsh \beta (A_{t})\eqd B_{t}$, hence to 
\eqref{;boug}.

This paper is a continuation of \cite{har22} and aims at 
providing an invariance similar to the above but 
without the presence of an independent element.
We retain $t>0$ fixed and define another  
anticipative path transformation $\ct \equiv \ct ^{t}$ 
on $C([0,t];\R )$ by 
\begin{align}\label{;defct}
 \ct (\phi )(s):=\tr _{2\phi _{t}}(\phi )(s),\quad 
 0\le s\le t,
\end{align}
for $\phi \in C([0,t];\R )$. Then, as in the theorem below, 
the law of Brownian motion is invariant under $\ct $, which, 
as far as we are concerned, has remained unseen until now.

\begin{thm}\label{;tmain}
 It holds that 
 \begin{align*}
   \{ \ct (B)(s)\} _{0\le s\le t}\eqd \{ B_{s}\} _{0\le s\le t}.
 \end{align*}
\end{thm}

Among other interesting properties, $\ct $ is an 
involution: $\ct \circ \ct =\id $, where $\id$ is the identity map 
on $C([0,t];\R )$; see \pref{;pct}\thetag{iv}. This property is 
consistent with \tref{;tmain} in the sense that, by the theorem,  
$\ct (B)\eqd (\ct \circ \ct )(B)=B$. Moreover, \tref{;tmain} is 
compatible with the time reversal of Brownian motion; 
see \rref{;rcompatible}\thetag{1}.

To facilitate the reader's understanding of the above theorem, 
it would also be informative to restate it in the following form: 

\let\temp\thethm 
\renewcommand{\thethm}{\ref{;tmain}$'$}
\begin{thm}\label{;tmaind}
It holds that 
\begin{align*}
 \left\{ 
 \frac{1}{A_{s}}+\frac{e^{2B_{t}}}{A_{t}}-\frac{1}{A_{t}}
 \right\} _{0<s\le t}
 \eqd 
 \left\{ 
 \frac{1}{A_{s}}
 \right\} _{0<s\le t}.
\end{align*}
\end{thm}

\let\thetheorem\temp 
\addtocounter{thm}{-1}

The expression of the left-hand side is due to \pref{;pct}\thetag{ii}. 
Notice that, thanks to the time reversal of Brownian motion,
\begin{align}\label{;qrev}
 \frac{e^{2B_{t}}}{A_{t}}\eqd \frac{1}{A_{t}};
\end{align}
see \rref{;rrev}. From the above identity in law, we see in particular that  
\begin{align*}
 \ex \!\left[ 
 \frac{e^{2B_{t}}}{A_{t}}-\frac{1}{A_{t}}
 \right] =0,
\end{align*} 
which is consistent with \tref{;tmaind}. 
As for the integrability of the random variables $1/A_{s},\,s>0$, 
we refer the reader to \cite{dmy00}. (In fact, it is true that  
$\ex \!\left[ 
\exp \left\{ \theta /(2A_{s})\right\} 
\right] <\infty $ for all $\theta <1$, whenever $s>0$; see, e.g., 
\cite[Lemma~4.2]{har22}.) We also note that identity \eqref{;qrev} 
is a particular case of \tref{;tmaind} evaluated at $s=t$.

As a corollary to \tref{;tmain}, by noting that $\ct (B)(t)=-B_{t}$ 
(see \pref{;pct}\thetag{i}), the Cameron--Martin formula 
immediately entails the following relation between Brownian motions 
with opposite drifts: 
\begin{align}\label{;opp}
 \bigl\{ \ct (\db{-\mu })(s)\bigr\} _{0\le s\le t}
 \eqd \bigl\{ \db{\mu }_{s}\bigr\} _{0\le s\le t}
\end{align}
for every $\mu \in \R $. Thanks to the fact that $\ct $ is an 
involution, we may put the above identity in law into a more 
symmetric form as stated in the corollary below.

\begin{cor}\label{;cmain}
 For every $\mu \in \R $, it holds that 
\begin{align*}
 \left\{ 
 \left( 
 \ct (\db{-\mu })(s),\,\db{-\mu }_{s}
 \right) 
 \right\} _{0\le s\le t}\eqd 
 \left\{ 
 \left( 
 \db{\mu }_{s},\,\ct (\db{\mu })(s)
 \right) 
 \right\} _{0\le s\le t}.
\end{align*}
\end{cor}

Let $\mu >0$ and set 
$\da{-\mu }_{\infty }:=\lim _{t\to \infty }\da{-\mu }_{t}$. 
Dufresne's identity in law  \cite[Proposition~4.4.4(b)]{duf} asserts that 
\begin{align}\label{;duf}
 \da{-\mu }_{\infty }
 &\eqd \frac{1}{2\ga _{\mu }},
\end{align}
where $\ga _{\mu }$ is a gamma random variable 
with parameter $\mu $:   
\begin{align*}
 \pr (\ga _{\mu }\in dx)
 &=\frac{1}{\Gamma (\mu )}x^{\mu -1}e^{-x}\,dx,\quad x>0. 
\end{align*}
Here $\Gamma (\cdot )$ is the gamma function.
Recall from Donati-Martin--Matsumoto--Yor \cite{dmy} 
the family $\{ T_{\al }\} _{\al \ge 0}$ of (non-anticipative) path 
transformations on $C([0,\infty );\R )$ defined by 
\begin{align}\label{;tra}
 T_{\al }(\phi )(s):=\phi _{s}-\log \left\{ 
 1+\al A_{s}(\phi )
 \right\} ,\quad s\ge 0,\ \phi \in C([0,\infty );\R ).
\end{align}
Another distributional relationship between 
$\db{\mu }$ and $\db{-\mu }$ is shown in 
Matsumoto--Yor \cite{myPIII}, which, in terms of the above 
notation, is stated as 

\begin{prop}[{\cite[Theorem~2.3]{myPIII}}]\label{;poppdg}
Suppose that $\mu >0$. Then it holds that 
\begin{align}\label{;oppdg}
 \left\{ 
 \left( 
 \db{-\mu }_{s}-\log \Biggl( 
 1-\frac{\da{-\mu }_{s}}{\da{-\mu }_{\infty }}
 \Biggr) ,\,\db{-\mu }_{s}
 \right) 
 \right\} _{s\ge 0}
 \eqd \left\{ 
 \left( 
 \db{\mu }_{s},\,T_{2\ga _{\mu }}(\db{\mu })(s)
 \right) 
 \right\} _{s\ge 0},
\end{align}
where, on the right-hand side, $\ga _{\mu }$ is independent of $B$.
\end{prop}

In \sref{;sppinv}, we see that relation \eqref{;oppdg} may be deduced 
from \cref{;cmain}. Combining the above proposition with \cref{;cmain}, 
we also obtain the following joint invariance of the law of 
Brownian motion with drift in the presence of an independent 
gamma random variable.
\begin{prop}\label{;pinv}
 Let $\mu >0$ and suppose that $\ga _{\mu }$ is independent 
 of $B$. Fix $t>0$.
 \begin{itemize}
  \item[\thetag{1}] 
  Denote 
  $\tr _{\log \{ e^{2\db{\mu }_{t}}\!/(1+2\ga _{\mu }\da{\mu }_{t})\} }(\db{\mu })$ 
  by $X^{1}$: 
  \begin{align*}
    X^{1}_{s}:=\db{\mu }_{s}-\log \left\{ 
   1+\frac{\da{\mu }_{s}}{\da{\mu }_{t}}\left( 
   \frac{e^{2\db{\mu }_{t}}}{1+2\ga _{\mu }\da{\mu }_{t}}-1
   \right) 
   \right\} ,\quad 0\le s\le t.
  \end{align*}
  Then it holds that 
  \begin{align}\label{;pinv1}
   \left\{ 
   (X^{1}_{s},\,\db{\mu }_{s})
   \right\} _{0\le s\le t}\eqd 
   \left\{ 
   (\db{\mu }_{s},\,X^{1}_{s})
   \right\} _{0\le s\le t}. 
  \end{align}

 \item[\thetag{2}] 
 Denote 
 $
 \tr _{\log \{ e^{2\db{-\mu }_{t}}+2\ga _{\mu }\da{-\mu }_{t}\} }(\db{-\mu })
 $
 by $X^{2}$:  
  \begin{align*}
   X^{2}_{s}:=\db{-\mu }_{s}-\log \left\{ 
   1+\frac{\da{-\mu }_{s}}{\da{-\mu }_{t}}\left( 
   e^{2\db{-\mu }_{t}}+2\ga _{\mu }\da{-\mu }_{t}-1
   \right) 
   \right\} ,\quad 0\le s\le t.
  \end{align*}
  Then it holds that 
  \begin{align}\label{;pinv2}
   \left\{ 
   (X^{2}_{s},\,\db{-\mu }_{s})
   \right\} _{0\le s\le t}\eqd 
   \left\{ 
   (\db{-\mu }_{s},\,X^{2}_{s})
   \right\} _{0\le s\le t}. 
  \end{align}
 \end{itemize}
\end{prop}

Identity \eqref{;pinv1} suggests that the process 
\begin{align*}
 \log \left\{ 
   1+\frac{\da{\mu }_{s}}{\da{\mu }_{t}}\left( 
   \frac{e^{2\db{\mu }_{t}}}{1+2\ga _{\mu }\da{\mu }_{t}}-1
   \right) 
   \right\} ,\quad 0\le s\le t,
\end{align*}
and hence its derivative with respect to $s$ as well, is symmetric, 
which is of independent interest; a similar 
remark also applies to \eqref{;pinv2}. Here we say that 
a real-valued continuous process $X=\{ X_{s}\} _{0\le s \le t}$ 
over the interval $[0,t]$ is symmetric if $-X\eqd X$.
Moreover, by letting $t\to \infty $ in \eqref{;pinv2}, we see 
in particular that the process  
\begin{align*}
 \db{-\mu }_{s}-\log \left\{ 
 1+\da{-\mu }_{s}\left( 
 2\ga _{\mu }-\frac{1}{\da{-\mu }_{\infty }}
 \right) 
 \right\} ,\quad s\ge 0,
\end{align*}
is identical in law with $\db{-\mu }$, which recovers 
\cite[Proposition~4.1]{har22}. A further remark on \pref{;pinv} is 
forwarded to \rref{;rpinv} after the end of the proof of the proposition.

The rest of the paper is organized as follows. In \sref{;sptrans}, 
we summarize properties of the transformation $\ct $ defined 
by \eqref{;defct}, which are referred to throughout the paper. 
In \sref{;sptmain}, we prove \tref{;tmain} and \cref{;cmain}: 
we give two proofs of \tref{;tmain} in \sssref{;ssp1tmain} and 
\ref{;ssp2tmain} while \cref{;cmain} is proven in \ssref{;sspcmain}. 
In \sref{;sppinv}, after giving a proof of \pref{;poppdg} by means of 
\cref{;cmain}, we prove \pref{;pinv}. In the final section, we provide 
some related results such as extensions of \tref{;tmain} deduced 
from \pref{;pinv}. 

For each $t>0$, we denote by $C([0,t];\R ^{2})$ the space of 
$\R ^{2}$-valued continuous functions over $[0,t]$. We equip 
each of the two spaces $C([0,t];\R )$ and $C([0,t];\R ^{2})$ with 
topology of uniform convergence; real-valued functionals on these 
spaces are said to be measurable if they are Borel-measurable 
with respect to those topologies. In the sequel, unless otherwise specified, 
$t>0$ is fixed and the path transformations 
$\ct $ and $\tr _{z},\,z\in \R $, refer to those on $C([0,t];\R )$ 
defined respectively by \eqref{;defct} and \eqref{;ttrans}.

%%%%%%%%% New Section %%%%%%%%%
\section{Properties of the transformation $\ct $}\label{;sptrans}

In this section, based on those of $\tr _{z},\,z\in \R $, investigated in 
\cite{har22}, we explore properties of the transformation $\ct $. 

Following the notation used in \cite{dmy} 
by Donati-Martin, Matsumoto and Yor, we define the path 
transformation $Z$ by 
\begin{align}\label{;defz}
 Z_{t}(\phi ):=e^{-\phi _{t}}A_{t}(\phi ),\quad t\ge 0,
\end{align}
for $\phi \in C([0,\infty );\R )$; as in the case of the exponential 
additive functional $\{ A_{t}\} _{t\ge 0}$ of $B$, we will simply write 
$Z_{t}$ for $Z_{t}(B)$ for each $t\ge 0$. 
Notice that the two transformations 
$A$ and $Z$ on $C([0,\infty );\R )$ are related via 
\begin{align}\label{;deria}
 \frac{d}{dt}\frac{1}{A_{t}(\phi )}
 =-\left\{ \frac{1}{Z_{t}(\phi )}\right\} ^{2},\quad t>0,
\end{align}
or, equivalently, for an arbitrarily fixed $t>0$, 
\begin{align}\label{;deriad}
 \frac{1}{A_{s}(\phi )}
 =\int _{s}^{t}\frac{du}{
 \left\{ Z_{u}(\phi )\right\} ^{2}}
 +\frac{e^{-\phi _{t}}}{Z_{t}(\phi )},\quad 0<s\le t, 
\end{align}
for any $\phi \in C([0,\infty );\R )$.

Again, we fix $t>0$ from now on, and 
let the path transformations $\ct $ and $\tr _{z},\,z\in \R $, 
be those on $C([0,t];\R )$.
Accordingly, the two transformations $A$ and $Z$ are 
restricted on 
$C([0,t];\R )$. We also consider the time reversal, which we 
denote by $R$, defined by 
\begin{align}\label{;trev}
 R(\phi )(s):=\phi _{t-s}-\phi _{t},\quad 0\le s\le t,
 \ \phi \in C([0,t];\R ).
\end{align}
It is well known that the law of 
$\{ B_{s}\} _{0\le s\le t}$ is invariant under $R$: 
\begin{align}\label{;itrev}
 \left\{ R(B)(s)\right\} _{0\le s\le t}
 \eqd 
 \{ B_{s}\} _{0\le s\le t}.
\end{align}
(Notice that the usual time reversal of Brownian motion 
refers to $\left\{ -R(B)(s)\right\} _{0\le s\le t}$ 
in our notation.) The following properties of 
$\tr _{z},\,z\in \R $, are investigated in \cite{har22}. 

\begin{lem}[{\cite[Proposition~2.1]{har22}}]\label{;lttrans}
We have the following \thetag{i}--\thetag{v}. 
\begin{itemize}
\item[\thetag{i}] For every $z\in \R $ and $\phi \in C([0,t];\R )$,
$\tr _{z}(\phi )(t)=\phi _{t}-z$.

\item[\thetag{ii}] For every $z\in \R $ and $\phi \in C([0,t];\R )$, 
\begin{align*}
 \frac{1}{A_{s}(\tr _{z}(\phi ))}=\frac{1}{A_{s}(\phi )}
 +\frac{e^{z}-1}{A_{t}(\phi )},\quad 0<s\le t;
\end{align*}
in particular, $A_{t}(\tr _{z}(\phi ))=e^{-z}A_{t}(\phi )$.

\item[\thetag{iii}] $Z\circ \tr _{z}=Z$ for any $z\in \R $.

\item[\thetag{iv}] (Semigroup property) 
$\tr _{z}\circ \tr _{z'}=\tr _{z+z'}$ for any $z,z'\in \R $; in particular, 
\begin{align*}
 \tr _{z}\circ \tr _{-z}=\tr _{0}=\id \quad \text{for any $z\in \R $},
\end{align*}
where $\id $ is the identity map on $C([0,t];\R )$
as in \sref{;intro}.

\item[\thetag{v}] For every $z\in \R $, 
$\tr _{z}\circ R\circ \tr _{z}=R$, and hence 
\begin{align*}
 R\circ \tr _{z}=\tr _{-z}\circ R.
\end{align*}

\end{itemize}
\end{lem}

The above properties may be verified by a direct computation;  see 
the proof of \cite[Proposition~2.1]{har22} for details. 
Since, among others, property~\thetag{iii} is frequently used in the 
paper, we provide its proof for the reader's convenience, 
together with a proof of \thetag{iv} that slightly differs from the one 
given in \cite{har22}. Notice that, for every $z\in \R $ and 
$\phi \in C([0,t];\R )$, we may rewrite the displayed identity 
in property~\thetag{ii} as 
\begin{align}\label{;deriaz}
 \frac{1}{A_{s}(\tr _{z}(\phi ))}
 =\int _{s}^{t}\frac{du}{
 \left\{ Z_{u}(\phi )\right\} ^{2}}
 +\frac{e^{-\phi _{t}+z}}{Z_{t}(\phi )},\quad 0<s\le t,
\end{align}
because of relation~\eqref{;deria} and the definition~\eqref{;defz} of $Z$.

\begin{proof}[Proofs of \thetag{iii} and \thetag{iv} of \lref{;lttrans}]
\thetag{iii} In view of relation~\eqref{;deria}, taking the derivative 
with respect to $s$ on both sides of \eqref{;deriaz} leads to 
\begin{align*}
 \left\{ Z_{s}(\tr _{z}(\phi ))\right\} ^{-2}
 =\left\{ Z_{s}(\phi )\right\} ^{-2},\quad 0<s\le t,
\end{align*}
which entails the claim by the positivity of $Z$.

\thetag{iv} It suffices to prove that, for each 
$\phi \in C([0,t];\R )$, 
\begin{align}\label{;epp1}
 A_{s}\bigl( (\tr _{z}\circ \tr _{z'})(\phi )\bigr)  
 =A_{s}(\tr _{z+z'}(\phi )),\quad 0\le s\le t; 
\end{align}
indeed, once this identity is proven, then taking the 
derivative with respect to $s$ on both sides verifies the claim. 
To this end, for every $0<s\le t$, successive use 
of property~\thetag{iii} in \eqref{;deriad} yields 
\begin{align*}
 \frac{1}{A_{s}\bigl( (\tr _{z}\circ \tr _{z'})(\phi )\bigr) }
 &=\int _{s}^{t}\frac{du}{\left\{ Z_{u}(\phi )\right\} ^{2}}
 +\frac{e^{-(\tr _{z}\circ \tr _{z'})(\phi )(t)}}{Z_{t}(\phi )}.
\end{align*}
Using property~\thetag{i} twice, we see that 
\begin{align*}
 -(\tr _{z}\circ \tr _{z'})(\phi )(t)&=-\tr _{z'}(\phi )(t)+z\\
 &=-\phi _{t}+z'+z,
\end{align*}
which proves \eqref{;epp1} by relation~\eqref{;deriaz}.
\end{proof}

Using the above lemma, we prove 

\begin{prop}\label{;pct}
The transformation $\ct $ has the following properties. 
\begin{itemize}
 \item[\thetag{i}] For every $\phi \in C([0,t];\R )$, 
 $\ct (\phi )(t)=-\phi _{t}$.
 
 \item[\thetag{ii}] For every $\phi \in C([0,t];\R )$, 
 \begin{align*}
  \frac{1}{A_{s}(\ct (\phi ))}=\frac{1}{A_{s}(\phi )}
  +\frac{e^{2\phi _{t}}-1}{A_{t}(\phi )},\quad 0<s\le t;
 \end{align*}
 in particular, $A_{t}(\ct (\phi ))=e^{-2\phi _{t}}A_{t}(\phi )$.
 
 \item[\thetag{iii}] $Z\circ \ct =Z$.
 
 \item[\thetag{iv}] For every $z\in \R $, 
 $\ct \circ \tr _{z}\circ \ct \circ \tr _{z}=\id $; in particular, 
 by taking $z=0$, 
 \begin{align*}
   \ct \circ \ct =\id .
 \end{align*}
 Moreover, 
 \begin{align}\label{;compt}
  (\ct \circ \tr _{z})(\phi )=\tr _{2\phi _{t}-z}(\phi ) 
 \end{align}
 for any $\phi \in C([0,t];\R )$. 
 
 \item[\thetag{v}] Restricted on the space of $\phi $'s vanishing 
 at the origin, $\ct $ is commutative with $R$; namely,
 for any $\phi \in C([0,t];\R )$ with $\phi _{0}=0$,
 \begin{align*}
   (R\circ \ct )(\phi )=(\ct \circ R)(\phi ).
 \end{align*}
\end{itemize}
\end{prop}

\begin{proof}
In view of the definition~\eqref{;defct} of $\ct $, properties~\thetag{i}, 
\thetag{ii} and \thetag{iii} follow by taking $z=2\phi _{t}$ in \thetag{i}, 
\thetag{ii} and \thetag{iii} of \lref{;lttrans}, respectively.

\thetag{iv} As in the proof of \lref{;lttrans}\thetag{iv}, in order to 
prove the first half of the assertion, it suffices to show that, 
for each $\phi \in C([0,t];\R )$, 
\begin{align}\label{;epp}
 A_{s}\bigl( (\ct \circ \tr _{z}\circ \ct \circ \tr _{z})(\phi )\bigr) 
 =A_{s}(\phi ),\quad 0\le s\le t. 
\end{align}
To this end, repeated use of 
property~\thetag{i} and \lref{;lttrans}\thetag{i} yields 
\begin{align*}
 (\ct \circ \tr _{z}\circ \ct \circ \tr _{z})(\phi )(t)
 &=-(\tr _{z}\circ \ct \circ \tr _{z})(\phi )(t)\\
 &=-(\ct \circ \tr _{z})(\phi )(t)+z\\
 &=\tr _{z}(\phi )(t)+z\\
 &=\phi _{t}, 
\end{align*}
where we used property~\thetag{i} for the first and third lines and 
\lref{;lttrans}\thetag{i} for the second and fourth. Then, 
for every $0<s\le t$, we have, by property~\thetag{iii} and 
\lref{;lttrans}\thetag{iii}, together with relation~\eqref{;deriad}, 
\begin{align*}
 \frac{1}{
 A_{s}\bigl( (\ct \circ \tr _{z}\circ \ct \circ \tr _{z})(\phi )\bigr)
 }
 &=\int _{s}^{t}\frac{du}{
 \left\{ Z_{u}(\phi )\right\} ^{2}}
 +\frac{
 e^{-(\ct \circ \tr _{z}\circ \ct \circ \tr _{z})(\phi )(t)}
 }{Z_{t}(\phi )}\\
 &=\int _{s}^{t}\frac{du}{
 \left\{ Z_{u}(\phi )\right\} ^{2}}
 +\frac{e^{-\phi _{t}}}{Z_{t}(\phi )}, 
\end{align*}
which agrees with $1/A_{s}(\phi )$ by \eqref{;deriad} again,  
proving \eqref{;epp}. To show the latter half, 
by \lref{;lttrans}\thetag{iv}, note that 
\begin{align*}
 (\ct \circ \tr _{z})(\phi )&=(\ct \circ \tr _{z})^{-1}(\phi )\\
 &=(\tr _{-z}\circ \ct )(\phi ),
\end{align*}
which is equal to $\tr _{-z+2\phi _{t}}(\phi )$ by the definition 
\eqref{;defct} of $\ct $ and the semigroup property 
in \lref{;lttrans}\thetag{iv}.

\thetag{v} By taking $z=2\phi _{t}$ in \lref{;lttrans}\thetag{v}, 
we have, for every $\phi \in C([0,t];\R )$, 
\begin{align*}
 (R\circ \ct )(\phi )=\tr _{-2\phi _{t}}(R(\phi )).
\end{align*} 
Since $R(\phi )(t)=\phi _{0}-\phi _{t}$ by the definition~\eqref{;trev} 
of $R$, we may rewrite the right-hand side of the above identity 
as $\tr _{2R(\phi )(t)-2\phi _{0}}(R(\phi ))$, and hence as 
$\ct (R(\phi ))$ when $\phi _{0}=0$.
\end{proof}

We give a remark on property~\thetag{v} in the above proposition.
\begin{rem}\label{;rcompatible}
\thetag{1} Property~\thetag{v} is compatible with \tref{;tmain} and 
the time reversal \eqref{;itrev} of Brownian motion: both $(R\circ \ct )(B)$ 
and $(\ct \circ R)(B)$ are identical in law with $\{ B_{s}\} _{0\le s\le t}$ 
by \tref{;tmain} and \eqref{;itrev}.

\thetag{2} Without the restriction $\phi _{0}=0$ in  
property~\thetag{v}, we have the relation  
\begin{align*}
 (R\circ \ct \circ R\circ R)(\phi )=(\ct \circ R)(\phi ) 
\end{align*}
for any $\phi \in C([0,t];\R )$. To see that, since 
$(R\circ R)(\phi )(s)=\phi _{s}-\phi _{0},\,0\le s\le t$, we have 
$(R\circ R)(\phi )(0)=0$ and $R\circ R\circ R=R$. Thereby 
we apply property~\thetag{v} to $(R\circ R)(\phi )$ to conclude 
that 
\begin{align*}
 (R\circ \ct )\bigl( (R\circ R)(\phi )\bigr) 
 &=(\ct \circ R)\bigl( (R\circ R)(\phi )\bigr) \\
 &=\ct \bigl(  (R\circ R\circ R)(\phi )\bigr) \\ 
 &=\ct (R(\phi )).
\end{align*}
\end{rem}

The composition of the transformations $\ct $ of different 
durations $t$ is also worth mentioning.

\begin{prop}\label{;ptcomp}
For every $u\ge 0$, it holds that, for any $\phi \in C([0,t+u];\R )$, 
\begin{align*}
 \ct ^{t}\bigl( 
 \ct ^{t+u}(\phi )
 \bigr) (s)=\tr ^{t}_{\phi _{t}+\ct ^{t+u}(\phi )(t)}(\phi )(s),\quad 0\le s\le t.
\end{align*}
\end{prop}

Since, in the case $u=0$, $\phi _{t}+\ct ^{t+u}(\phi )(t)=0$ by 
\pref{;pct}\thetag{i}, the above proposition gives an extension of 
the property $\ct \circ \ct =\id $ in \pref{;pct}\thetag{iv}.

\begin{proof}[Proof of \pref{;ptcomp}]
 For every $0\le s\le t$, we have, by the definition \eqref{;defct} 
 of $\ct ^{t}$,
 \begin{align}\label{;qptcomp}
  \ct ^{t}\bigl( 
 \ct ^{t+u}(\phi )
 \bigr) (s)=\tr ^{t}_{2\ct ^{t+u}(\phi )(t)}\bigl( 
 \ct ^{t+u}(\phi )
 \bigr) (s).
 \end{align}
 It is also readily seen from the definition~\eqref{;ttrans} of 
 $\{ \tr ^{t}_{z}\} _{z\in \R }$ that 
 \begin{align*}
  \ct ^{t+u}(\phi )(v)=\tr ^{t}_{\phi _{t}-\ct ^{t+u}(\phi )(t)}(\phi )(v),
  \quad 0\le v\le t.
 \end{align*}
 Therefore, by the semigroup property in \lref{;lttrans}\thetag{iv}, 
 the right-hand side of \eqref{;qptcomp} is equal to  
 \begin{align*}
  \tr ^{t}_{2\ct ^{t+u}(\phi )(t)+\phi _{t}-\ct ^{t+u}(\phi )(t)}
  (\phi ) (s),
 \end{align*}
 which proves the claim.
\end{proof}

We end this section with a remark concerning identity \eqref{;qrev}.

\begin{rem}\label{;rrev}
 By \eqref{;itrev}, it holds that 
 \begin{align*}
 \left( 
 e^{B_{t}},\,A_{t}
 \right) \eqd 
 \left( e^{-B_{t}},\,e^{-2B_{t}}A_{t}\right) .
 \end{align*}
 Indeed, the left-hand side is identical in law with 
 \begin{align*}
  \left( e^{R(B)(t)},\,A_{t}(R(B))\right) ,
 \end{align*}
 which is equal to the right-hand side by the definition \eqref{;trev} 
 of $R$. The above identity in law may also be deduced from 
 \tref{;tmain} in such a way that 
 \begin{equation*}
  \begin{split}
  \left( 
  e^{B_{t}},\,A_{t}
  \right) 
  &\eqd \left( e^{\ct (B)(t)},\,A_{t}(\ct (B))\right) \\
  &=\left( e^{-B_{t}},\,e^{-2B_{t}}A_{t}\right) ,
 \end{split}
 \end{equation*}
 thanks to properties~\thetag{i} and \thetag{ii} in \pref{;pct} 
 for the second line.
\end{rem}

%%%%%%%%% New Section %%%%%%%%%
\section{Proofs of \tref{;tmain} and \cref{;cmain}}\label{;sptmain}

In this section, we prove \tref{;tmain} and \cref{;cmain}. 
We give two proofs of \tref{;tmain}; although the second proof 
is more direct, the first one, which utilizes \lref{;lbb} below, is of 
interest in its own right. 

\subsection{First proof of \tref{;tmain}}\label{;ssp1tmain}

For every $x\in \R $, we denote by 
$b^{x}=\{ b^{x}_{s}\} _{0\le s \le t}$ a Brownian bridge of 
duration $t$ starting from $0$ and ending at $x$. Notice that 
\begin{align}\label{;bb}
 b^{x}\eqd \left\{ b^{0}_{s}+\frac{x}{t}s\right\} _{0\le s \le t}.
\end{align}
We begin with the following lemma.

\begin{lem}\label{;lbb}
For every $x,z\in \R $, we have, for any nonnegative measurable 
functional $F$ on $C([0,t];\R )$, 
\begin{equation}\label{;eqlbb}
 \begin{split}
 &\exp \left( 
 -\frac{z^{2}}{2t}
 \right) \ex \!\left[ 
 \exp \left\{ 
 -\frac{\cosh x}{Z_{t}(b^{z})}
 \right\} F\bigl( 
 \tr _{z-x}(b^{z})
 \bigr) 
 \right] \\
 &=\exp \left( 
 -\frac{x^{2}}{2t}
 \right) \ex \!\left[ 
 \exp \left\{ 
 -\frac{\cosh z}{Z_{t}(b^{x})}
 \right\} F(b^{x})
 \right] .
 \end{split}
\end{equation}
\end{lem}

Observe that, in the left-hand side,  by \lref{;lttrans}\thetag{i},  
$\tr _{z-x}(b^{z})(t)=z-(z-x)=x$, which is consistent with the 
right-hand side, meaning that $b^{x}_{t}=x$.

Fix $z\in \R $ arbitrarily. Recall from \cite[Theorem~1.2]{har22} 
that, for any nonnegative measurable functional $F$ on $C([0,t];\R )$,
\begin{align*}
 \ex \!\left[ 
 F\bigl( \tr _{z}(B)\bigr) 
 \right] 
 =\ex \!\left[ 
 \exp \left\{ 
 \frac{\cosh B_{t}-\cosh (z+B_{t})}{Z_{t}}
 \right\} F(B)
 \right] .
\end{align*}
We replace $F$ by a functional of the form 
\begin{align*}
 \exp \left\{ 
 -\frac{\cosh \phi _{t}}{Z_{t}(\phi )} 
 \right\} F(\phi ),\quad 
 \phi \in C([0,t];\R ),
\end{align*}
with $F$ a nonnegative measurable functional again; then, by 
properties~\thetag{i} and \thetag{iii} in \lref{;lttrans}, the above 
identity turns into 
\begin{align*}
 \ex \!\left[ 
 \exp \left\{ 
 -\frac{\cosh (B_{t}-z)}{Z_{t}}
 \right\} F\bigl( \tr _{z}(B)\bigr) 
 \right] 
 =\ex \!\left[ 
 \exp \left\{ 
 -\frac{\cosh (z+B_{t})}{Z_{t}}
 \right\} F(B) 
 \right] .
\end{align*}

\begin{proof}[Proof of \lref{;lbb}]
In the last equation, we substitute into $F$ a functional of the 
form 
\begin{align*}
 f(\phi _{t})F(\phi ),\quad \phi \in C([0,t];\R ),
\end{align*}
where 
$f:\R \to [0,\infty )$ is a measurable function, and we suppose, 
to begin with, that $F:C([0,t];\R )\to [0,\infty )$ is bounded 
and continuous. Then we have, by noting \lref{;lttrans}\thetag{i} 
as to the left-hand side,  
\begin{align*}
 &\int _{\R }\frac{dx}{\sqrt{2\pi t}}\exp \left( 
 -\frac{x^{2}}{2t}
 \right) f(x-z)\ex \!\left[ 
 \exp \left\{ 
 -\frac{\cosh (x-z)}{Z_{t}(b^{x})}
 \right\} F\bigl( \tr _{z}(b^{x})\bigr) 
 \right] \\
 &=\int _{\R }\frac{dx}{\sqrt{2\pi t}}\exp \left( 
 -\frac{x^{2}}{2t}
 \right) f(x)\ex \!\left[ 
 \exp \left\{ 
 -\frac{\cosh (z+x)}{Z_{t}(b^{x})}
 \right\} F(b^{x})
 \right] .
\end{align*}
By translation, the left-hand side is 
rewritten as 
\begin{align*}
 \int _{\R }\frac{dx}{\sqrt{2\pi t}}\exp \left\{  
 -\frac{(x+z)^{2}}{2t}
 \right\}  f(x)\ex \!\left[ 
 \exp \left\{ 
 -\frac{\cosh x}{Z_{t}(b^{x+z})}
 \right\} F\bigl( \tr _{z}(b^{x+z})\bigr) 
 \right] .
\end{align*}
Therefore, because of the arbitrariness of $f$, we have, 
for a.e.\ $x\in \R $,
\begin{align*}
 &\exp \left\{  
 -\frac{(x+z)^{2}}{2t}
 \right\}  \ex \!\left[ 
 \exp \left\{ 
 -\frac{\cosh x}{Z_{t}(b^{x+z})}
 \right\} F\bigl( \tr _{z}(b^{x+z})\bigr) 
 \right] \\
 &=\exp \left( 
 -\frac{x^{2}}{2t}
 \right) \ex \!\left[ 
 \exp \left\{ 
 -\frac{\cosh (z+x)}{Z_{t}(b^{x})}
 \right\} F(b^{x})
 \right] .
\end{align*}
In view of \eqref{;bb} and thanks to the boundedness and continuity 
of $F$, the bounded convergence theorem entails that 
both sides are continuous in $x$. Hence the last 
equality holds for any $x$, and for any $z$ as well since 
$z$ was arbitrarily fixed. Consequently, replacing $z$ by 
$z-x$ proves identity \eqref{;eqlbb} 
when $F$ is bounded and continuous. Then, density and 
monotone class arguments extend $F$ to any 
nonnegative measurable function as claimed.
\end{proof}

We are in a position to prove \tref{;tmain}.

\begin{proof}[First proof of \tref{;tmain}]
 For each $x\in \R $, in \eqref{;eqlbb}, we substitute $-x$ into $z$ 
 and replace $F$ by a functional of the form 
 \begin{align*}
  \exp \left\{ 
  \frac{\cosh x}{Z_{t}(\phi )}
  \right\} F(\phi ),\quad \phi \in C([0,t];\R ),
 \end{align*}
 with $F$ an arbitrary nonnegative 
 measurable functional. Then, thanks to \lref{;lttrans}\thetag{iii} 
 as to the left-hand side, identity \eqref{;eqlbb} becomes 
 \begin{align}\label{;tbb}
  \ex \!\left[ 
  F\bigl( 
  \tr _{-2x}(b^{-x})
  \bigr) 
  \right] 
  =\ex \!\left[ 
  F(b^{x})
  \right] .
 \end{align}
 Integrating both sides with respect to $P(B_{t}\in dx)$ 
 over $\R $ and using the symmetry $-B\eqd B$ on the 
 left-hand side leads to the conclusion.
\end{proof}

\subsection{Second proof of \tref{;tmain}}\label{;ssp2tmain}

We denote by $\{ \Z _{s}\} _{s\ge 0}$ the natural filtration of the 
process $\{ Z_{s}\} _{s\ge 0}$. The proof of \tref{;tmain} given below 
hinges upon the observation that the conditional law 
of $B_{t}$ given $\Z _{t}$ is symmetric as in the next lemma; 
see also \rref{;rlcsym} at the end of this subsection.

\begin{lem}\label{;lcsym}
 It holds that 
 \begin{align*}
 \left( 
 B_{t},\,\{ Z_{s}\} _{0\le s\le t}
 \right) \eqd 
 \left( 
 -B_{t},\,\{ Z_{s}\} _{0\le s\le t}
 \right) .
 \end{align*}
\end{lem}

\begin{proof}
Let $f:\R \to \R $ be a bounded measurable function and 
$F:C([0,t];\R )\to \R $ a bounded measurable functional.
We have, conditionally on $\Z _{t}$,
\begin{align}\label{;cz1}
 \ex \!\left[ 
 f(B_{t})F(Z)
 \right] =
 \ex \!\left[ 
 \ex \!\left[ 
 f(B_{t})\rmid | \Z _{t}
 \right] F(Z)
 \right] .
\end{align}
We know from \cite[Proposition~1.7]{myPI} that the conditional 
expectation in the right-hand side is equal a.s.\ to 
\begin{align}\label{;cz2}
 \ex \!\left[ 
 f(z_{u})
 \right] \big| _{u=1/Z_{t}}, 
\end{align}
where, for each $u>0$, $z_{u}$ refers to a real-valued random variable 
whose law is given by 
\begin{align*}
 \frac{1}{2K_{0}(u)}e^{-u\cosh x}\,dx,\quad x\in \R .
\end{align*}
Here $K_{0}$ is the modified Bessel function of the third kind 
(or the Macdonald function) of order $0$. Since the above law is 
symmetric, we have 
$
 \ex \!\left[ 
 f(z_{u})
 \right] =
 \ex \!\left[ 
 f(-z_{u})
 \right] 
$ 
for every $u>0$, and hence from \eqref{;cz1} and \eqref{;cz2}, 
\begin{align*}
 \ex \!\left[ 
 f(B_{t})F(Z)
 \right] =
 \ex \!\left[ 
 f(-B_{t})F(Z)
 \right] .
\end{align*}
As $f$ and $F$ are arbitrary, we have the claim.
\end{proof}

By using the above conditional symmetry of $B_{t}$, the second 
proof of \tref{;tmain} proceeds as follows.

\begin{proof}[Second proof of \tref{;tmain}]
If we have proven 
\begin{align}\label{;qp2}
 \left\{ 
 A_{s}(\ct (B))
 \right\} _{0\le s\le t}
 \eqd \left\{ 
 A_{s}
 \right\} _{0\le s\le t},
\end{align}
then, taking the derivative with respect to $s$ on each side of 
the above identity leads to the conclusion. 
To this end, notice that, for each $0<s\le t$, 
\begin{align*}
 \frac{1}{A_{s}(\ct (B))}&=\int _{s}^{t}\frac{du}{Z_{u}^{2}}+\frac{e^{-B_{t}}}{Z_{t}}
 +\frac{e^{2B_{t}}-1}{A_{t}}\\
 &=\int _{s}^{t}\frac{du}{Z_{u}^{2}}+\frac{e^{B_{t}}}{Z_{t}},
\end{align*}
where we have used \pref{;pct}\thetag{ii} as well as relation \eqref{;deriad} 
for the first line and the definition of $Z_{t}$ (see \eqref{;defz}) for the 
second. By \lref{;lcsym}, the last expression entails that 
\begin{align*}
 \left\{ 
 \frac{1}{A_{s}(\ct (B))}
 \right\} _{0<s\le t}
 &\eqd \left\{ 
 \int _{s}^{t}\frac{du}{Z_{u}^{2}}+\frac{e^{-B_{t}}}{Z_{t}}
 \right\} _{0<s\le t}\\
 &=\left\{ 
 \frac{1}{A_{s}}
 \right\} _{0<s\le t},
\end{align*}
which proves \eqref{;qp2}. Here we used relation \eqref{;deriad} 
again for the last equality.
\end{proof}

The above proof confirms that \tref{;tmaind} is indeed equivalent to 
\tref{;tmain} as indicated just above \tref{;tmaind}.

\begin{rem}\label{;rlcsym}
We may associate \lref{;lcsym} with the fact \cite[Theorem~1.6\thetag{ii}]{myPI} 
that, 
for any $\mu \in \R $, 
 \begin{align}\label{;symz}
  \left\{ 
  Z_{s}(\db{\mu })
  \right\} _{s\ge 0}
  \eqd 
  \left\{ 
  Z_{s}(\db{-\mu })
  \right\} _{s\ge 0},
 \end{align}
in such a way that, by the Cameron--Martin formula, 
\begin{align*}
 \ex \!\left[ 
 e^{\mu B_{t}}F(Z)
 \right] 
 =\ex\! \left[ 
 e^{-\mu B_{t}}F(Z)
 \right] 
\end{align*}
for every bounded measurable functional $F$ on $C([0,t];\R )$. Then 
the injectivity of the Mellin transform entails the lemma. Identity \eqref{;symz} 
may be explained by the identity in law between the second coordinates 
in \eqref{;oppdg}, and by the fact that $Z\circ T_{\al }=Z$ for every $\al \ge 0$ 
(\cite[Proposition~2.1\thetag{iii}]{dmy}).
\end{rem}

\subsection{Proof of \cref{;cmain}}\label{;sspcmain}

First we show identity \eqref{;opp}. Let $F$ be a nonnegative 
measurable functional on $C([0,t];\R )$. For every $\mu \in \R $, 
it holds that, by \tref{;tmain},  
\begin{align*}
 \ex \!\left[ 
 F\bigl( 
 \ct (B)
 \bigr) e^{\mu \ct (B)(t)}
 \right] 
 =\ex \!\left[ 
 F(B)e^{\mu B_{t}}
 \right] .
\end{align*}
Multiplying both sides by $e^{-\mu ^{2}t/2}$ and noting 
$\ct (B)(t)=-B_{t}$ by \pref{;pct}\thetag{i}, we have, by the 
Cameron--Martin formula, 
\begin{align*}
 \ex \!\left[ 
 F\bigl( 
 \ct (\db{-\mu })
 \bigr) 
 \right] 
 =\ex \!\left[ 
 F(\db{\mu })
 \right] ,
\end{align*}
which verifies \eqref{;opp}.

\begin{proof}[Proof of \cref{;cmain}]
It follows from \eqref{;opp} that 
\begin{align*}
 \left\{ 
 \left( 
 \ct (\db{-\mu })(s),\,(\ct \circ \ct )\bigl( \db{-\mu }\bigr) (s)
 \right) 
 \right\} _{0\le s\le t}
 \eqd \left\{ 
 \left( 
 \db{\mu }_{s},\,\ct \bigl( 
 \db{\mu }
 \bigr) (s)
 \right) 
 \right\} _{0\le s\le t}.
\end{align*}
Since $\ct \circ \ct =\mathrm{Id}$ as stated in \pref{;pct}\thetag{iv}, 
we have the claim.
\end{proof}

Note that \cref{;cmain} may also be obtained 
by integrating both sides of \eqref{;tbb} with respect to 
the probability measure 
\begin{align*}
 \frac{1}{\sqrt{2\pi t}}\exp \left\{ 
 -\frac{(x-\mu t)^{2}}{2t}
 \right\} dx
\end{align*}
over $\R $, for any drift $\mu \in \R $.

%%%%%%%%% New Section %%%%%%%%%
\section{Proof of \pref{;pinv}}\label{;sppinv}

This section is devoted to the proof of \pref{;pinv};  
in order to make the paper self-contained as much as possible, 
we also give a proof of \pref{;poppdg}, which will be done by 
using \cref{;cmain}. Except for the proof of \pref{;poppdg}, 
we suppose that $t>0$ is fixed.

We begin this section with the 

\begin{proof}[Proof of \pref{;poppdg} via \cref{;cmain}]
Let $\mu >0$ and fix $u>0$ arbitrarily. Set the process 
$X=\{ X_{v}\} _{v\ge 0}$ by 
\begin{align*}
 X_{v}:=\db{\mu }_{v+u}-\db{\mu }_{u}, 
\end{align*}
which has the same law as $\db{\mu }$ and 
is independent of $\{ B_{s}\} _{0\le s\le u}$.
We let $t>0$ be such that $u<t$. Then, by the definition \eqref{;defct} 
of $\ct $, \cref{;cmain} entails that 
the two-dimensional process 
\begin{align*}
 \left( 
 \db{\mu }_{s},\,
 \db{\mu }_{s}-\log \left\{ 
 1+\da{\mu }_{s}
 \frac{
 e^{2\db{\mu }_{u}}\!e^{2X_{t-u}}-1
 }
 {
 \da{\mu }_{u}+e^{2\db{\mu }_{u}}\!A_{t-u}(X)
 }
 \right\} 
 \right) 
 ,\quad 0\le s\le u,
\end{align*}
is identical in law with 
$\bigl\{ 
\bigl( 
\ct (\db{-\mu })(s),\,\db{-\mu }_{s}
\bigr) 
\bigr\} _{0\le s\le u}$. 
Rewrite 
\begin{align}\label{;rv}
 \frac{
 e^{2\db{\mu }_{u}}\!e^{2X_{t-u}}-1
 }
 {
 \da{\mu }_{u}+e^{2\db{\mu }_{u}}\!A_{t-u}(X)
 }
 =\frac{
 e^{2\db{\mu }_{u}}-e^{-2X_{t-u}}
 }
 {
 e^{-2X_{t-u}}\da{\mu }_{u}+e^{2\db{\mu }_{u}}\!e^{-2X_{t-u}}A_{t-u}(X)
 },
\end{align}
and observe that $e^{-2X_{t-u}}\to 0$ a.s.\ as $t\to \infty $ 
because $\mu >0$. Moreover, by the time reversal of 
Brownian motion, 
\begin{align*}
 e^{-2X_{t-u}}A_{t-u}(X)\eqd \da{-\mu }_{t-u},
\end{align*}
which converges in law to $1/(2\ga _{\mu })$ as $t\to \infty $ 
by Dufresne's identity \eqref{;duf}. Hence, owing to the 
independence of $\{ B_{s}\} _{0\le s\le u}$ and $X$, 
the pair of the process $\bigl\{ \db{\mu }_{s}\bigr\} _{0\le s \le u}$ and 
the random variable \eqref{;rv} jointly converges in law to 
that of $\bigl\{ \db{\mu }_{s}\bigr\} _{0\le s\le u}$ and $2\ga _{\mu }$, 
with $\ga _{\mu }$ being independent of $B$. Therefore we obtain 
the identity in law 
\begin{align*}
 &\left\{ 
 \left( 
 \db{\mu }_{s},\,
 \db{\mu }_{s}-\log \left\{ 
 1+2\ga _{\mu }\da{\mu }_{s}\right\} 
 \right) 
 \right\} _{0\le s\le u}\\
 &\eqd 
 \left\{ 
 \left( 
 \db{-\mu }_{s}-\log \Biggl( 
 1-\frac{\da{-\mu }_{s}}{\da{-\mu }_{\infty }}
 \Biggr) ,\,\db{-\mu }_{s}
 \right) 
 \right\} _{0\le s\le u},
\end{align*}
where the expression of the first coordinate in the 
right-hand side is due to the definition \eqref{;defct} of 
$\ct $ and the fact that $e^{2\db{-\mu }_{t}}\to 0$ a.s.\ as 
$t\to \infty $.
Since $u>0$ is arbitrary, the last identity in law proves the proposition 
by the definition~\eqref{;tra} of $\{ T_{\al }\} _{\al \ge 0}$.
\end{proof}

We proceed to the proof of \pref{;pinv}. For every $\al \ge 0$, 
note that the expression 
\begin{align}\label{;expr1}
 \ct (\T _{\al }(\phi ))(s),\quad 0\le s\le t,
\end{align}
makes sense for any $\phi \in C([0,t];\R )$, and so does 
the expression 
\begin{align}\label{;expr2}
 \T _{\al }(\ct (\phi ))(s),\quad 0\le s\le t,
\end{align}
because $\T _{\al }$ is a non-anticipative transformation. 
Notice that, for any $\phi \in C([0,t];\R )$, $T_{\al }(\phi )$ 
is represented as 
\begin{align}\label{;tat}
 T_{\al }(\phi )(s)=\tr _{\log \{ 1+\al A_{t}(\phi )\} }(\phi )(s),\quad 
 0\le s\le t,
\end{align}
by the definition \eqref{;ttrans} of $\{ \tr _{z}\} _{z\in \R }$; indeed,  
\begin{align*}
 \tr _{\log \{ 1+\al A_{t}(\phi )\} }(\phi )(s)
 &=\phi _{s}-\log \left\{ 
 1+\frac{A_{s}(\phi )}{A_{t}(\phi )}\bigl( 1+\al A_{t}(\phi )-1\bigr) 
 \right\} \\
 &=T_{\al }(\phi )(s).
\end{align*}

\begin{lem}\label{;lexpr}
\thetag{1} Expression~\eqref{;expr1} admits the 
representation 
\begin{align*}
 \ct (\T _{\al }(\phi ))(s)
 =\tr _{\log \{ e^{2\phi _{t}}/(1+\al A_{t}(\phi ))\} }(\phi )(s),\quad 
 0\le s\le t.
\end{align*}

\noindent
\thetag{2} Expression~\eqref{;expr2} admits the 
representation 
\begin{align*}
 \T _{\al }(\ct (\phi ))(s)
 =\tr _{\log \{e^{2\phi _{t}}+\al A_{t}(\phi )\}}(\phi )(s),
 \quad 0\le s\le t.
\end{align*}
\end{lem}

\begin{proof}
Fix $0\le s\le t$ arbitrarily.

\thetag{1} By \eqref{;tat} and relation~\eqref{;compt} in 
\pref{;pct}\thetag{iv},
\begin{align*}
 \ct \bigl( 
 T_{\al }(\phi )
 \bigr) (s)&=\ct \bigl( 
 \tr _{\log \{ 1+\al A_{t}(\phi )\} }(\phi )
 \bigr) (s)\\
 &=\tr _{2\phi _{t}-\log \{ 1+\al A_{t}(\phi )\} }(\phi )(s),
\end{align*}
which proves the claim.

\thetag{2} By \eqref{;tat}, 
\begin{align*}
 T_{\al }( \ct (\phi ))(s)
 =\tr _{\log \{ 1+\al A_{t}(\ct (\phi ))\} }(\ct (\phi ))(s),
\end{align*}
which, by \pref{;pct}\thetag{i}, is rewritten as 
\begin{align*}
 \tr _{\log \{ 1+\al e^{-2\phi _{t}}A_{t}(\phi )\} }(\ct (\phi ))(s)
 =\tr _{\log \{ 1+\al e^{-2\phi _{t}}A_{t}(\phi )\} +2\phi _{t}}(\phi )(s),
\end{align*}
proving \thetag{2}. Here the last equality is due to the semigroup 
property in \lref{;lttrans}\thetag{iv} and the definition \eqref{;defct} 
of $\ct $.
\end{proof}

Using the above lemma, we prove \pref{;pinv}. First observe that, 
for each fixed $t>0$, by the Markov property of Brownian motion 
and Dufresne's identity \eqref{;duf}, we have the identity in law 
\begin{align}\label{;markov}
 \left( 
 \da{-\mu }_{\infty },\,\bigl\{ 
 \db{-\mu }_{s}
 \bigr\} _{0\le s\le t}
 \right) 
 \eqd 
 \left( 
 \da{-\mu }_{t}+e^{2\db{-\mu }_{t}}\!/(2\ga _{\mu }),\,\bigl\{ 
 \db{-\mu }_{s}
 \bigr\} _{0\le s\le t}
 \right) ,
\end{align}
where, in the right-hand side, $\ga _{\mu }$ is independent of $B$.
This observation entails that, in view of the definition \eqref{;ttrans} of 
$\{ \tr _{z}\} _{z\in \R }$, \pref{;poppdg} may be restated as 
the identity in law between the two two-dimensional processes 
\begin{align}\label{;td1}
 \left( 
 \tr _{
 2\db{-\mu }_{t}-\log \{ e^{2\db{-\mu }_{t}}+2\ga _{\mu }\da{-\mu }_{t}\} 
 }(\db{-\mu })(s),\,\db{-\mu }_{s}
 \right) ,\quad 0\le s\le t,
\end{align}
and 
\begin{align}\label{;td2}
 \left( 
 \db{\mu }_{s},\,T_{2\ga _{\mu }}(\db{\mu })(s)
 \right) ,\quad 0\le s\le t,
\end{align}
because of the equality 
\begin{align*}
 -\frac{\da{-\mu }_{s}}{\da{-\mu }_{t}+e^{2\db{-\mu }_{t}}\!/(2\ga _{\mu })}
 =\frac{\da{-\mu }_{s}}{\da{-\mu }_{t}}\left( 
 \frac{e^{2\db{-\mu }_{t}}}{e^{2\db{-\mu }_{t}}+2\ga _{\mu }\da{-\mu }_{t}}-1
 \right) 
\end{align*}
for every $0\le s\le t$.

\begin{proof}[Proof of \pref{;pinv}]
\thetag{1} Since the process 
\begin{align*}
 \left( 
 \db{\mu }_{s},\,\ct \bigl( T_{2\ga _{\mu }}(\db{\mu })\bigr) (s)
 \right) ,\quad 0\le s\le t,
\end{align*}
is nothing but the right-hand side of \eqref{;pinv1} by 
\lref{;lexpr}\thetag{1}, it suffices to prove, in view of \eqref{;td1} 
and \eqref{;td2}, that the process 
\begin{align}\label{;td1d}
 \left( 
 \tr _{
 2\db{-\mu }_{t}-\log \{ e^{2\db{-\mu }_{t}}+2\ga _{\mu }\da{-\mu }_{t}\} 
 }(\db{-\mu })(s),\,\ct (\db{-\mu })(s)
 \right) ,\quad 0\le s\le t,
\end{align}
is identical in law with the left-hand side of \eqref{;pinv1}. 
To this end, for each $0\le s\le t$, we rewrite the first coordinate in 
\eqref{;td1d} in such a way that 
\begin{align*}
 \tr _{-\log \{ 1+2\ga _{\mu }A_{t}(\ct (\db{-\mu }))\} }\bigl( 
 (\ct \circ \ct )(\db{-\mu })
 \bigr) (s)
\end{align*}
by \pref{;pct}\thetag{ii} and by the fact that $\ct $ is an 
involution (\pref{;pct}\thetag{iv}). Then, by \cref{;cmain}, we see that \eqref{;td1d} 
is identical in law with 
\begin{align*}
 \left( 
 \tr _{-\log \{ 1+2\ga _{\mu }\da{\mu }_{t}\} }\bigl( 
 \ct (\db{\mu })
 \bigr) (s),\,\db{\mu }_{s}
 \right) ,\quad 0\le s\le t,
\end{align*}
which coincides with the left-hand side of \eqref{;pinv1} by 
the semigroup property of $\{ \tr _{z}\} _{z\in \R }$ in 
\lref{;lttrans}\thetag{iv} and the definition \eqref{;defct} of $\ct $.

\thetag{2} Since, for each $0\le s\le t$, we may express the first 
coordinate in \eqref{;td1} as 
\begin{align*}
 \bigl( 
 \ct \circ 
 \tr _{\log \{ e^{2\db{-\mu }_{t}}+2\ga _{\mu }\da{-\mu }_{t}\}}
 \bigr)(\db{-\mu })(s) 
\end{align*}
due to \eqref{;compt} in \pref{;pct}\thetag{iv}, it suffices to 
prove, in view of \eqref{;td1} and \eqref{;td2}, that the process 
\begin{align*}
 \left( 
 \ct (\db{\mu })(s),\,T_{2\ga _{\mu }}(\db{\mu })(s)
 \right) ,\quad 0\le s\le t,
\end{align*}
is identical in law with the right-hand side of \eqref{;pinv2} 
owing to the fact that $\ct $ is an involution. 
Rewriting the last displayed process as 
\begin{align*}
 \left( 
 \ct (\db{\mu })(s),\,T_{2\ga _{\mu }}
 \bigl( (\ct \circ \ct )(\db{\mu })\bigr) (s)
 \right) ,\quad 0\le s\le t,
\end{align*}
we see that it is identical in law with 
\begin{align*}
 \left( 
 \db{-\mu }_{s},\,T_{2\ga _{\mu }}\bigl( 
 \ct (\db{-\mu })
 \bigr) (s)
 \right) ,\quad 0\le s\le t,
\end{align*}
by \cref{;cmain}, which, thanks to \lref{;lexpr}\thetag{2}, coincides 
with the right-hand side of \eqref{;pinv2} as claimed.
\end{proof}

\begin{rem}\label{;rpinv}
By extracting the gamma variable $\ga _{\mu }$ from 
the first coordinate in the left-hand side, 
identity \eqref{;pinv1} is equivalently rephrased as the joint identity 
in law 
\begin{align*}
 \left( 
 \ga _{\mu },\,\left\{ 
 (X^{1}_{s},\,\db{\mu }_{s})
 \right\} _{0\le s\le t}
 \right) \eqd 
 \left( 
 \frac{\ga _{\mu }e^{2\db{\mu }_{t}}}{1+2\ga _{\mu }\da{\mu }_{t}},\,
 \left\{ 
 (\db{\mu }_{s},\,X^{1}_{s})
 \right\} _{0\le s\le t}
 \right) ;
\end{align*}
similarly, identity \eqref{;pinv2} is equivalent to 
\begin{align*}
 \left( 
 \ga _{\mu },\,\left\{ 
 (X^{2}_{s},\,\db{-\mu }_{s})
 \right\} _{0\le s\le t}
 \right) \eqd 
 \left( 
 \frac{\ga _{\mu }}{e^{2\db{-\mu }_{t}}+2\ga _{\mu }\da{-\mu }_{t}},\,
 \left\{ 
 (\db{-\mu }_{s},\,X^{2}_{s})
 \right\} _{0\le s\le t}
 \right) .
\end{align*}
In each of the above two identities, the identity in law between the 
first components may be explained in the following manner: 
\begin{align*}
 \frac{\ga _{\mu }e^{2\db{\mu }_{t}}}{1+2\ga _{\mu }\da{\mu }_{t}}
 &=\frac{\ga _{\mu }}{
 e^{-2\db{\mu }_{t}}+2\ga _{\mu }e^{-2\db{\mu }_{t}}\da{\mu }_{t}
 }\\
 &\eqd \frac{\ga _{\mu }}{e^{2\db{-\mu }_{t}}+2\ga _{\mu }\da{-\mu }_{t}}\\
 &\eqd \frac{1}{2\da{-\mu }_{\infty }},
\end{align*}
which is identical in law with $\ga _{\mu }$ by 
Dufresne's identity \eqref{;duf}. Here the second line is due to the 
time reversal of Brownian motion (see \eqref{;itrev}) and the 
third line is nothing but the identity in law between the 
first components in \eqref{;markov}.

\end{rem}

Combining \cref{;cmain} and \pref{;ptcomp} enables us to obtain in part 
a generalization of \pref{;pinv}. Let $\cB =\{ \cB _{s}\} _{s\ge 0}$ be 
a one-dimensional standard Brownian motion that is independent of $B$. 
For every drift $\mu \in \R $, we denote 
\begin{align*}
 \dca{\mu }_{s}=A_{s}\bigl( \cB ^{(\mu )}\bigr) ,\quad s\ge 0,
\end{align*}
with 
$
\dcb{\mu }=\bigl\{ \dcb{\mu }_{s}\equiv \cB _{s}+\mu s\bigr\} _{s\ge 0}
$. 
Then we have the following distributional invariance of 
$\db{\mu }$ in the presence of the independent 
element $\dcb{\mu }$:

\begin{prop}\label{;pdii}
 For every $\mu \in \R $, it holds that, for any $u\ge 0$, 
 the process 
 \begin{align}\label{;qpdii}
  \db{\mu }_{s}-\log \left\{ 
  1+\frac{\da{\mu}_{s}}{\da{\mu }_{t}}\left( 
  \frac{
  e^{2\db{\mu }_{t}}\!\dca{\mu }_{u}+\da{\mu }_{t}
  }
  {
  \dca{\mu }_{u}+e^{2\dcb{\mu }_{u}}\!\da{\mu }_{t}
  }
  -1\right) 
  \right\} ,\quad 0\le s\le t,
 \end{align}
 is identical in law with $\bigl\{ \db{\mu }_{s}\bigr\} _{0\le s\le t}$.
\end{prop}

\begin{proof}
 Since there is nothing to prove in the case $u=0$, we let 
 $u>0$. In view of \cref{;cmain} and \pref{;ptcomp}, the process 
 \begin{align*}
  \db{\mu }_{s}-\log \left( 
  1+\frac{\da{\mu }_{s}}{\da{\mu }_{t}}\left\{ 
  \frac{e^{2\db{\mu }_{t}}}
  {
  1+\bigl( \da{\mu }_{t}\!/\da{\mu }_{t+u}\bigr) 
  \bigl( e^{2\db{\mu }_{t+u}}-1\bigr) 
  }-1
  \right\} 
  \right) ,\quad 0\le s\le t,
 \end{align*}
 has the same law as $\bigl\{ \db{\mu }_{s}\bigr\} _{0\le s\le t}$. 
 Because Brownian motion has independent increments, we see that 
 the pair of $\bigl\{ \db{\mu }_{s}\bigr\} _{0\le s\le t}$ and the 
 random variable 
 \begin{align*}
  1+\frac{\da{\mu }_{t}}{\da{\mu }_{t+u}}\bigl( e^{2\db{\mu }_{t+u}}-1\bigr) 
 \end{align*}
 is identical in law with that of $\bigl\{ \db{\mu }_{s}\bigr\} _{0\le s\le t}$ 
 and 
 \begin{align*}
  1+\frac{\da{\mu }_{t}}{\da{\mu }_{t}+e^{2\db{\mu }_{t}}\!\dca{\mu }_{u}}
  \bigl( 
  e^{2\dcb{\mu }_{u}}\!e^{2\db{\mu }_{t}}-1
  \bigr) .
 \end{align*}
 Since the last displayed random variable equals 
 \begin{align*}
  e^{2\db{\mu }_{t}}\!
  \frac{
  \dca{\mu }_{u}+\da{\mu }_{t}e^{2\dcb{\mu }_{u}}
  }
  {
  \da{\mu }_{t}+e^{2\db{\mu }_{t}}\!\dca{\mu }_{u}
  },
 \end{align*}
 we have the claim.
\end{proof}

To see that the above proposition generalizes \pref{;pinv} partly,
notice that, when $\mu >0$,
\begin{align*}
 \frac{
  e^{2\db{\mu }_{t}}\!\dca{\mu }_{u}+\da{\mu }_{t}
  }
  {
  \dca{\mu }_{u}+e^{2\dcb{\mu }_{u}}\!\da{\mu }_{t}
  }
  =\frac{
  e^{2\db{\mu }_{t}}\!e^{-2\dcb{\mu }_{u}}\!\dca{\mu }_{u}+e^{-2\dcb{\mu }_{u}}\!\da{\mu }_{t}
  }
  {
  e^{-2\dcb{\mu }_{u}}\!\dca{\mu }_{u}+\da{\mu }_{t}
  }
\end{align*}
in \eqref{;qpdii} converges in law to 
\begin{align*}
 \frac{e^{2\db{\mu }_{t}}\!/(2\ga _{\mu })}{1/(2\ga _{\mu })+\da{\mu }_{t}}
 =\frac{e^{2\db{\mu }_{t}}}{1+2\ga _{\mu }\da{\mu }_{t}}
\end{align*}
as $u\to \infty $, owing to Dufresne's identity~\eqref{;duf}, 
with $\ga _{\mu }$ being independent of $\db{\mu }$; recall the 
reasoning after equation~\eqref{;rv} in the proof of \pref{;poppdg}. 
The case $\mu <0$ is similar but treated more readily.

%%%%%%%%% New Section %%%%%%%%%
\section{Some related results and extensions}\label{;ssrre}

In this section, we provide some results related to those 
introduced above, including extensions of \tref{;tmain}.

To begin with, for every point $a\in \R $ and 
$\phi \in C([0,\infty );\R )$, we denote by 
$\tau _{a}(\phi )$ the first hitting time of $\phi $ to the level $a$:
\begin{align*}
 \tau _{a}(\phi ):=\inf \{ s\ge 0;\,\phi _{s}=a\},
\end{align*}
with the convention $\inf \emptyset =\infty $. 
Let $\beta =\{ \beta (s)\} _{s\ge 0}$ and 
$\hB=\{ \hB _{s}\} _{s\ge 0}$ be two one-dimensional standard 
Brownian motions that are independent of $B$. 
Let $x\in \R $ be fixed and denote
\begin{align}\label{;rvt}
 \tau ^{x}:=\tau _{\cosh (x+B_{t})}(\hB ^{(\cosh x/Z_{t})})
\end{align}
for simplicity. As recalled in \sref{;intro}, it is shown in 
\cite{har22} that the process 
\begin{align*}
 \tr _{x+B_{t}-\argsh (e^{B_{t}}\sinh x +\beta (A_{t}))}(B)(s),\quad 
 0\le s\le t, 
\end{align*}
is a Brownian motion; more precisely, 
what in fact we have proven is 

\begin{prop}[{\cite[Theorem~3.1]{har22}}]\label{;psdi0}
Under the above setting, we have the following for every $x\in \R $: 
\begin{itemize}
 \item[\thetag{i}] the pair of the process 
\begin{align*}
 \tr _{x+B_{t}-\argsh (e^{B_{t}}\sinh x +\beta (A_{t}))}(B)(s),\quad 
 0\le s\le t, 
\end{align*}
and the random variable $A_{t}$ is identical in law with that of 
$\{ B_{s}\} _{0\le s\le t}$ and $\tau ^{x}$;
 
 \item[\thetag{ii}] the pair of the process 
\begin{align*}
 \tr _{\log (A_{t}/\tau ^{x})}(B)(s),\quad 0\le s\le t,
\end{align*}
and the random variable $\log (A_{t}/\tau ^{x})$ is identical in law 
with that of $\{ B_{s}\} _{0\le s\le t}$ and 
\begin{align*}
 \argsh \bigl( e^{B_{t}}\sinh x +\beta (A_{t})\bigr) -x-B_{t}.
\end{align*}
\end{itemize}
\end{prop}

By virtue of \tref{;tmain}, similar distributional identities to the above 
hold as in the following proposition.

\begin{prop}\label{;psdi}
Under the same setting as in \pref{;psdi0}, we have the following for 
every $x\in \R $:

\begin{itemize}
 \item[\thetag{i}] the pair of the process 
 \begin{align*}
  \tr _{\argsh (e^{B_{t}}\sinh x +\beta (A_{t}))-x+B_{t}}(B)(s),\quad 0\le s\le t,
 \end{align*}
 and the random variable $A_{t}$ is identical in law with that of 
 $\{ B_{s}\} _{0\le s\le t}$ and 
 \begin{align*}
  \tau _{\cosh (x-B_{t})}(\hB ^{(\cosh x/Z_{t})});
 \end{align*}
 
 \item[\thetag{ii}] the pair of the process 
 \begin{align*}
  \tr _{\log (e^{2B_{t}}\tau ^{x}/A_{t})}(B)(s),\quad 
 0\le s\le t,
 \end{align*}
 and the random variable $\log (A_{t}/\tau ^{x})$ is identical in law with 
 that of $\{ B_{s}\} _{0\le s\le t}$ and 
 \begin{align*}
  \argsh \bigl( e^{-B_{t}}\sinh x+e^{-B_{t}}\beta (A_{t})\bigr) -x+B_{t}.
 \end{align*}
\end{itemize}
\end{prop}

\begin{proof}
\thetag{i} Because of \pref{;psdi0}\thetag{i} and the relation that, 
for each $0\le s\le t$,  
\begin{align*}
 \ct \bigl( 
 \tr _{x+B_{t}-\argsh (e^{B_{t}}\sinh x +\beta (A_{t}))}(B)
 \bigr) (s)
 =\tr _{\argsh (e^{B_{t}}\sinh x +\beta (A_{t}))-x+B_{t}}(B)(s)
\end{align*}
due to relation~\eqref{;compt} in 
\pref{;pct}\thetag{iv}, it suffices to show that 
\begin{align}\label{;qpsdi1}
 \left( 
 \{ \ct (B)(s)\} _{0\le s\le t},\,\tau ^{x}
 \right) \eqd 
 \bigl( 
 \{ B_{s}\} _{0\le s\le t},\,\tau _{\cosh (x-B_{t})}(\hB ^{(\cosh x/Z_{t})})
 \bigr) .
\end{align}
To this end, notice that, by the definition \eqref{;rvt} of $\tau ^{x}$ and 
properties~\thetag{i} and \thetag{iii} in \pref{;pct}, 
\begin{align*}
 \tau ^{x}=\tau _{\cosh (x-\ct (B)(t))}(\hB ^{(\cosh x/Z_{t}(\ct (B)))}).
\end{align*}
As a consequence, by \tref{;tmain} and the independence of $B$ and $\hB $, 
we have the claimed identity \eqref{;qpsdi1}.

\thetag{ii} Similarly to \thetag{i}, by virtue of \pref{;psdi0}\thetag{ii}, 
it suffices to prove that 
\begin{equation}\label{;qpsdi2}
 \begin{split}
  &\left( 
  \{ \ct (B)(s)\} _{0\le s\le t},\,
  \argsh \bigl( e^{B_{t}}\sinh x +\beta (A_{t})\bigr) -x-B_{t}
  \right) \\
  &\eqd \left( 
  \{ B_{s}\} _{0\le s\le t},\,
  \argsh \bigl( e^{-B_{t}}\sinh x +e^{-B_{t}}\beta (A_{t})\bigr) 
  -x+B_{t}
  \right) .
 \end{split}
\end{equation}
Note that, by properties~\thetag{i} and \thetag{ii} in \pref{;pct}, 
the second component in the left-hand side may be written as 
\begin{align*}
 \argsh \bigl( e^{-\ct (B)(t)}\sinh x +\beta 
 \bigl( e^{-2\ct (B)(t)}A_{t}(\ct (B))\bigr) \bigr) -x+\ct (B)(t),
\end{align*}
which entails that the left-hand side of the claimed identity~\eqref{;qpsdi2} 
is identical in law with 
\begin{align*}
 \left( \{ B_{s}\} _{0\le s\le t},\,
  \argsh \bigl( e^{-B_{t}}\sinh x+\beta (e^{-2B_{t}}A_{t})\bigr) -x+B_{t}\right) 
\end{align*}
owing to \tref{;tmain}. This verifies \eqref{;qpsdi2} by the scaling property 
of Brownian motion and the independence of $B$ and $\beta $.
\end{proof}

\begin{rem}\label{;rsecond}
The identity in law between the second components in \eqref{;qpsdi1} 
may also be explained by means of the time reversal \eqref{;itrev}; 
indeed, by \eqref{;itrev}, the random variable $\tau ^{x}$ has the same 
law as 
\begin{align*}
 \tau _{\cosh (x+R(B)(t))}(\hB ^{(\cosh x/Z_{t}(R(B)))}),
\end{align*}
which is nothing but the second component in the right-hand side of 
\eqref{;qpsdi1} because $R(B)(t)=-B_{t}$ and 
\begin{align*}
 Z_{t}(R(B))&=e^{-R(B)(t)}A_{t}(R(B))\\
 &=e^{B_{t}}e^{-2B_{t}}A_{t}\\
 &=Z_{t}
\end{align*}
by the definition \eqref{;defz} of the transformation $Z$.
\end{rem}

Next we will see that \pref{;pinv} may be rephrased as

\begin{prop}\label{;pinvr}
For every $x\ge 0$, it holds that, for any nonnegative measurable 
functional $F$ on $C([0,t];\R ^{2})$,
\begin{align}
 &\ex \!\left[ 
 F\bigl( 
 \tr _{\log \{ e^{2B_{t}}/(1+2xA_{t})\} }(B),B \bigr) 
 \right] \notag \\
 &=\ex \!\left[ 
 \frac{e^{2B_{t}}}{e^{2B_{t}}-2xA_{t}}\exp \left( 
 x-\frac{x}{e^{2B_{t}}-2xA_{t}}
  \right) 
 F\bigl( 
 B,\tr _{\log (e^{2B_{t}}-2xA_{t})}(B)\bigr) ;\,\frac{e^{2B_{t}}}{2A_{t}}>x
 \right]  \label{;rephr1}
\intertext{and}
 &\ex \!\left[ 
 F\bigl( 
 \tr _{\log (e^{2B_{t}}+2xA_{t})}(B),B\bigr) 
 \right] \notag \\
 &=\ex \!\left[ 
 \frac{1}{1-2xA_{t}}\exp \left( 
 x-\frac{xe^{2B_{t}}}{1-2xA_{t}}
  \right) 
 F\bigl( 
 B,\tr _{\log \{ e^{2B_{t}}/(1-2xA_{t})\} }(B)
 \bigr) ;\,\frac{1}{2A_{t}}>x
 \right] . \label{;rephr2}
\end{align}
\end{prop}

The above proposition extends \tref{;tmain} in the sense that 
the theorem is recovered by taking $x=0$: 
\begin{align}\label{;mu0}
 \left\{ 
 \left( 
 \ct (B)(s),\,B_{s}
 \right) 
 \right\} _{0\le s\le t}
 \eqd \left\{ 
 \left( 
 B_{s},\,\ct (B)(s)
 \right) 
 \right\} _{0\le s\le t},
\end{align}
that is, the case $\mu =0$ in \cref{;cmain}.

\begin{proof}[Proof of \pref{;pinvr}]
Since the two identities \eqref{;rephr1} and 
\eqref{;rephr2} are proven in the same way, we only give a proof 
for the latter.

For every nonnegative measurable functional $F$ on 
$C([0,t];\R ^{2})$, it follows from \pref{;pinv}\thetag{2} that, by the 
Cameron--Martin formula, 
\begin{align*}
 \ex \!\left[ 
 F\bigl( \tr _{\log (e^{2B_{t}}+2\ga _{\mu }A_{t})}(B),B\bigr) 
 e^{-\mu B_{t}}
 \right] 
 =\ex \!\left[ 
 F\bigl( 
 B,\tr _{\log (e^{2B_{t}}+2\ga _{\mu }A_{t})}(B)
 \bigr) e^{-\mu B_{t}}
 \right] .
\end{align*}
Replacing $F$ by a nonnegative functional of the form 
$F(\phi ^{1},\phi ^{2})e^{\mu \phi ^{2}_{t}},\,(\phi ^{1},\phi ^{2})
\in C([0,t];\R ^{2})$, we have 
\begin{align}\label{;repl}
 \ex \!\left[ 
 F\bigl( \tr _{\log (e^{2B_{t}}+2\ga _{\mu }A_{t})}(B),B\bigr) 
 \right] 
 =\ex \!\left[ 
 F\bigl( 
 B,\tr _{\log (e^{2B_{t}}+2\ga _{\mu }A_{t})}(B)
 \bigr) \bigl( e^{2B_{t}}+2\ga _{\mu }A_{t}\bigr) ^{-\mu }
 \right] ,
\end{align}
where the expression of the right-hand side is due to 
property~\thetag{i} in \lref{;lttrans} and, in addition to the nonnegativity, 
we assume, for the time being, that $F$ is bounded and continuous. 
By the independence of 
$B$ and $\ga _{\mu }$, and by Fubini's theorem, the right-hand side 
of the above identity is rewritten as 
\begin{align*}
 \frac{1}{\Gamma (\mu )}
 \ex \!\left[ 
 \int _{0}^{\infty }dy\,y^{\mu -1}e^{-y}
 F\bigl( 
 B,\tr _{\log (e^{2B_{t}}+2yA_{t})}(B)
 \bigr) \bigl( e^{2B_{t}}+2yA_{t}\bigr) ^{-\mu }
 \right] ,
\end{align*}
which, by changing the variables with 
$y/(e^{2B_{t}}+2yA_{t}) =x,\,0<x<1/(2A_{t})$, is further 
rewritten as 
\begin{align*}
 \frac{1}{\Gamma (\mu )}\ex \!\left[ 
 \int _{0}^{1/(2A_{t})}
 dx\,\frac{x^{\mu -1}}{1-2xA_{t}}\exp \left( 
 -\frac{xe^{2B_{t}}}{1-2xA_{t}}
 \right) 
 F\bigl( 
 B,\tr _{\log \{ e^{2B_{t}}/(1-2xA_{t})\} }(B)
 \bigr) 
 \right] .
\end{align*}
 Therefore, identity~\eqref{;repl} is rephrased as 
\begin{align*}
 &\int _{0}^{\infty }dx\,x^{\mu -1}e^{-x}
 \ex \!\left[ 
 F\bigl( \tr _{\log (e^{2B_{t}}+2xA_{t})}(B),B\bigr) 
 \right] \\
 &=\int _{0}^{\infty }dx\,x^{\mu -1}
 \ex \!\left[ 
 \frac{1}{1-2xA_{t}}\exp \left( 
 -\frac{xe^{2B_{t}}}{1-2xA_{t}}
 \right) 
 F\bigl( 
 B,\tr _{\log \{ e^{2B_{t}}/(1-2xA_{t})\} }(B)
 \bigr) ;\,\frac{1}{2A_{t}}>x
 \right] ,
\end{align*}
where we used the independence of $B$ and $\ga _{\mu }$ 
for the left-hand side, and Fubini's theorem again for the right-hand 
side. Since the above identity holds for any $\mu >0$, the injectivity 
of the Mellin transform entails that, for a.e.\ $x\ge 0$, 
\begin{equation}\label{;rephr2d}
 \begin{split}
 &e^{-x}\ex \!\left[ 
 F\bigl( \tr _{\log (e^{2B_{t}}+2xA_{t})}(B),B\bigr) 
 \right] \\
 &=\ex \!\left[ 
 \frac{1}{1-2xA_{t}}\exp \left( 
 -\frac{xe^{2B_{t}}}{1-2xA_{t}}
 \right) 
 F\bigl( 
 B,\tr _{\log \{ e^{2B_{t}}/(1-2xA_{t})\} }(B)
 \bigr) ;\,\frac{1}{2A_{t}}>x
 \right] .
 \end{split} 
\end{equation}
It is clear that, in view of the definition \eqref{;ttrans} of 
$\{ \tr _{z}\} _{z\in \R }$, the left-hand side is continuous in $x$ by the 
bounded convergence theorem, because of the fact that 
$F$ is assumed to be bounded and continuous. 
On the other hand, notice that, in the right-hand side, the integrand 
\begin{align*}
 \frac{1}{1-2xA_{t}}\exp \left( 
 -\frac{xe^{2B_{t}}}{1-2xA_{t}}
 \right) 
 F\bigl( 
 B,\tr _{\log \{ e^{2B_{t}}/(1-2xA_{t})\} }(B)
 \bigr) 
\end{align*}
for $0\le x<1/(2A_{t})$ and $0$ otherwise, is continuous in $x$ 
a.s., thanks to the boundedness and continuity of $F$; moreover, 
it is bounded from above by the integrable random variable 
\begin{align*}
 M\max \left\{ 2e^{-2B_{t}}A_{t},1\right\} ,
\end{align*}
where 
$
M:=\sup \{ F(\phi ^{1},\phi ^{2});\,(\phi ^{1},\phi ^{2})\in C([0,t];\R ^{2})\} 
$. 
Therefore the right-hand side of 
\eqref{;rephr2d} also gives rise to a continuous function in $x$ by the 
dominated convergence theorem, ensuring that 
identity~\eqref{;rephr2d} holds for all $x\ge 0$. 
Standard arguments of density and monotone class then 
extend $F$ to any nonnegative measurable functional and 
complete the proof of \eqref{;rephr2}.
\end{proof}

We give two remarks on \pref{;pinvr}.

\begin{rem}
 \pref{;pinvr} suggests that we have the following two relations 
 for any $\psi \in C([0,t];\R )$ such that 
 $e^{2\psi _{t}}-2xA_{t}(\psi )>0$ for the former and that 
 $1-2xA_{t}(\psi )>0$ for the latter:  
 \begin{align}
  \tr _{\log\{ e^{2\phi _{t}}/(1+2xA_{t}(\phi ))\} }(\phi )
  \big| _{\phi =\tr _{\log \{ e^{2\psi _{t}}-2xA_{t}(\psi )\} }(\psi )}&=\psi ;
  \label{;pcac1}\\
  \tr _{\log\{ e^{2\phi _{t}}+2xA_{t}(\phi )\} }(\phi )
  \big| _{\phi =\tr _{\log \{ e^{2\psi _{t}}/(1-2xA_{t}(\psi ))\} }(\psi )}
  &=\psi .
  \label{;pcac2}
 \end{align}
 Indeed, a direct computation verifies these two relations. 
 As for \eqref{;pcac1}, 
 observe that, by properties~\thetag{i} and \thetag{ii} in 
 \lref{;lttrans}, 
 \begin{align*}
  &\frac{e^{2\phi _{t}}}{1+2xA_{t}(\phi )}
  \bigg| _{\phi =\tr _{\log \{ e^{2\psi _{t}}-2xA_{t}(\psi )\} }(\psi )}\\
  &=\left\{ 
  \frac{e^{\psi _{t}}}{e^{2\psi _{t}}-2xA_{t}(\psi )}
  \right\} ^{2}\times 
  \frac{1}{1+2xA_{t}(\psi )/\{ e^{2\psi _{t}}-2xA_{t}(\psi )\} }\\
  &=\frac{1}{e^{2\psi _{t}}-2xA_{t}(\psi )}.
 \end{align*}
 Therefore the left-hand side of \eqref{;pcac1} is written as 
 \begin{align*}
  \tr _{-\log \{ e^{2\psi _{t}}-2xA_{t}(\psi )\}}(\phi )
  \big| _{\phi =\tr _{\log \{ e^{2\psi _{t}}-2xA_{t}(\psi )\}}(\psi )},
 \end{align*}
 which, by \lref{;lttrans}\thetag{iv}, equals $\psi $ as claimed 
 in \eqref{;pcac1}. Similarly, as for \eqref{;pcac2}, 
 \begin{align*}
  &\bigl\{ 
  e^{2\phi _{t}}+2xA_{t}(\phi )
  \bigr\} 
  \big|_{\phi =\tr _{\log \{ e^{2\psi _{t}}/(1-2xA_{t}(\psi ))\} }(\psi )}\\
  &=\left\{ 
  \frac{1-2xA_{t}(\psi )}{e^{\psi _{t}}}
  \right\} ^{2}+2x\frac{1-2xA_{t}(\psi )}{e^{2\psi _{t}}}A_{t}(\psi )\\
  &=\frac{1-2xA_{t}(\psi )}{e^{2\psi _{t}}},
 \end{align*}
 and hence the left-hand side of \eqref{;pcac2} is 
 written as 
 \begin{align*}
  \tr _{-\log \{ e^{2\psi _{t}}/(1-2xA_{t}(\psi ))\} }(\phi )
  \big| _{\phi =\tr _{\log \{ e^{2\psi _{t}}/(1-2xA_{t}(\psi ))\} }(\psi )},
 \end{align*}
 which equals $\psi $ and verifies \eqref{;pcac2}.
\end{rem}

\begin{rem}\label{;rcac}
The two relations \eqref{;rephr1} and \eqref{;rephr2} are equivalent and 
they are related via \tref{;tmain} (or, more precisely, identity \eqref{;mu0}). 
For instance, to see that the former entails the 
latter, we replace $F$ by a functional of the form 
$
F\bigl( \ct (\phi ^{1}),\ct (\phi ^{2})\bigr) ,\,
(\phi ^{1},\phi ^{2})\in C([0,t];\R ^{2})
$.  
Then, in view of 
\lref{;lexpr}\thetag{1}, the left-hand side of \eqref{;rephr1} turns into 
\begin{equation}\label{;rcac1}
\begin{split}
 \ex \!\left[ 
 F\bigl( 
 \ct (\ct (T_{2x}(B))),\ct (B)
 \bigr) 
 \right] 
 &=\ex \!\left[ 
 F\bigl( 
 T_{2x}(B),\ct (B)
 \bigr) 
 \right] \\
 &=\ex \!\left[ 
 F\bigl( 
 T_{2x}(\ct (B)),B
 \bigr) 
 \right] ,
\end{split}
\end{equation}
which agrees with the left-hand side of \eqref{;rephr2} 
thanks to \lref{;lexpr}\thetag{2}. Here we used 
the property $\ct \circ \ct =\id $ in \pref{;pct}\thetag{iv} 
for the first line and \eqref{;mu0} for the second. 
On the other hand, as for the right-hand side, observe the relation 
\begin{align}\label{;obs1}
 \left( e^{B_{t}},\,A_{t}\right) 
 =\left( e^{-\ct (B)(t)},\,e^{-2\ct (B)(t)}A_{t}(\ct (B))\right) 
\end{align}
and the fact that 
\begin{equation}\label{;obs2}
\begin{split}
 \ct (\tr _{\log (e^{2B_{t}}-2xA_{t})}(B))(s)
 &=\tr _{2B_{t}-\log (e^{2B_{t}}-2xA_{t})}(B)(s)\\
 &=\tr _{-\log \{ 1-2xA_{t}(\ct (B))\} }(B)(s)
\end{split}
\end{equation}
for each $0\le s\le t$. Relation \eqref{;obs1} is due to properties 
\thetag{i} and \thetag{ii} in \pref{;pct}, while in \eqref{;obs2}, we used 
\eqref{;compt} for the first line and \eqref{;obs1} for the second. 
Then, thanks to these two observations \eqref{;obs1} and \eqref{;obs2}, 
the above replacement of $F$ turns the right-hand side of 
\eqref{;rephr1} into 
\begin{multline*}
 \ex \Biggl[ 
 \frac{1}{1-2xA_{t}(\ct (B))}\exp \left\{ 
 x-\frac{xe^{2\ct (B)(t)}}{1-2xA_{t}(\ct (B))}
 \right\} F\bigl( \ct (B),\tr _{-\log \{ 1-2xA_{t}(\ct (B))\} }(B)\bigr);\\
 \frac{1}{2A_{t}(\ct (B))}>x 
 \Biggr] ,
\end{multline*}
which is equal to 
\begin{align*}
 \ex \!\left[ 
 \frac{1}{1-2xA_{t}}\exp \left( 
 x-\frac{xe^{2B_{t}}}{1-2xA_{t}}
  \right) 
 F\bigl( 
 B,\tr _{-\log (1-2xA_{t})}(\ct (B))
 \bigr) ;\,\frac{1}{2A_{t}}>x
 \right] 
\end{align*}
by \eqref{;mu0}. Since 
\begin{align*}
 \tr _{-\log (1-2xA_{t})}(\ct (B))
 =\tr _{\log \{ e^{2B_{t}}/(1-2xA_{t})\} }(B)
\end{align*}
by the definition \eqref{;defct} of $\ct $ and 
the semigroup property of $\{ \tr _{z}\} _{z\in \R }$ (\lref{;lttrans}\thetag{iv}), 
the last expectation agrees with the right-hand side of \eqref{;rephr2}.
\end{rem}

Taking $F$ as a functional of the first coordinate,  for every fixed $x\in \R $,
we obtain from \pref{;pinvr} Girsanov-type formulas for the two  
anticipative transforms $\tr _{\log \{ e^{2B_{t}}/(1+2xA_{t})\} }(B)$ and 
$\tr _{\log (e^{2B_{t}}+2xA_{t})}(B)$, which is of interest 
from the viewpoint of Malliavin calculus as they would provide examples 
in which the associated Fredholm determinants are explicitly calculated; 
we refer to \cite[Section~6]{har22} and references cited therein 
in this respect.

We conclude this paper with a remark concerning the non-anticipative 
transform $T_{2x}(B)$, as treated in \eqref{;rcac1}, of the 
Brownian motion $B$ up to time $t$.
\begin{rem}
In view of \eqref{;rcac1}, it is also revealed in \rref{;rcac} that, 
for each $x\ge 0$,  
\begin{align*}
 &\ex \!\left[ 
 F\bigl( 
 T_{2x}(B),\ct (B)
 \bigr) 
 \right] \\
 &=\ex \!\left[ 
 \frac{1}{1-2xA_{t}}\exp \left( 
 x-\frac{xe^{2B_{t}}}{1-2xA_{t}}
  \right) 
 F\bigl( B,\tr _{\log \{ e^{2B_{t}}/(1-2xA_{t})\} }(B)\bigr) ;\,
 \frac{1}{2A_{t}}>x
 \right] .
\end{align*}
Replacing $F$ by a functional of the form 
$F\bigl( \phi ^{1},\ct (\phi ^{2})\bigr) ,\,
(\phi ^{1},\phi ^{2})\in C([0,t];\R ^{2})$, 
we obtain the relation 
\begin{align*}
 &\ex \!\left[ 
 F\bigl( 
 T_{2x}(B),B
 \bigr) 
 \right] \\
 &=\ex \!\left[ 
 \frac{1}{1-2xA_{t}}\exp \left( 
 x-\frac{xe^{2B_{t}}}{1-2xA_{t}}
  \right) 
 F\bigl( B,\tr _{\log (1-2xA_{t})}(B)\bigr) ;\,
 \frac{1}{2A_{t}}>x
 \right] .
\end{align*}
If, for every $\mu \in \R $, we substitute into $F$ a functional 
of the form 
\begin{align*}
 F(\phi ^{1},\phi ^{2})e^{\mu \phi ^{2}_{t}},\quad 
 (\phi ^{1},\phi ^{2})\in C([0,t];\R ^{2}),
\end{align*}
then we have, by the Cameron--Martin formula, 
\begin{align*}
 &\ex \!\left[ 
 F\bigl( T_{2x}(\db{\mu }),\db{\mu }\bigr) 
 \right] \\
 &=\ex \!\left[ 
 \frac{1}{\bigl\{   
 1-2x\da{\mu }_{t}
 \bigr\} ^{\mu +1}}\exp \left\{ 
 x-\frac{xe^{2\db{\mu }_{t}}}{1-2x\da{\mu }_{t}}
 \right\} F\bigl( \db{\mu },\tr _{\log \{ 1-2x\da{\mu }_{t}\} }(\db{\mu })\bigr) 
 ;\,\frac{1}{2\da{\mu }_{t}}>x
 \right] 
\end{align*}
for any $x\ge 0$, which extends \cite[Theorem~1.5]{dmy}, in particular,  
to the case of negative drifts $\mu $. We also note that, by \eqref{;tat}, 
the left-hand side may be expressed as 
\begin{align*}
 \ex \!\left[ 
 F\bigl( \tr _{\log \{ 1+2x\da{\mu }_{t}\} }(\db{\mu }),\db{\mu }\bigr) 
 \right] 
\end{align*}
in terms of the transformations $\tr _{z},\,z\in \R $. 
\end{rem}

%%%%%%%%% References %%%%%%%%%

\end{document}